\documentclass[12pt,twoside, reqno]{amsart}
\usepackage{amsmath, amsthm, amscd, amsfonts, amssymb, graphicx}
\newtheorem{thm}{Theorem}[section]
\newtheorem{cor}[thm]{Corollary}

\newtheorem{defn}[thm]{Definition}

\numberwithin{equation}{section}
\newtheorem{exm}[thm]{Example}

\begin{document}
\begin{center}
{\bf{Application of a generalization of Darbo's fixed point theorem on an integral equation involving the weighted fractional integral with respect to another function 
 }}
\vspace{.5cm}
\\    Sudip Deb  ,  Anupam Das  

\vspace{.2cm}
Department of Mathematics, Cotton University, Panbazar, Guwahati-781001, Assam, India\\

Email:  math.sudip2022@gmail.com  ;  mth2191003\rule{0.21cm}{0.5pt}sudip@cottonuniversity.ac.in ;\\
 math.anupam@gmail.com 
\end{center}

\title{}
\author{}
\begin{abstract}  
In this study ,we display a generalization of  Darbo's fixed point theorem ,by using the use of a freshly made contraction operator and that we use  to study the solvability of an integral equation involving the weighted fractional integral with respect to another function in Banach space. Also, we cosider some appropriate example for verifying our results.  

\vskip 0.5cm


\vskip 0.5cm

\textbf{MSC subject classification no :}  26A33 ; 45G05 ; 45G10 ; 47H10
\end{abstract}

\maketitle
\pagestyle{myheadings}
\markboth{\rightline {\scriptsize Deb
}}
         {\leftline{\scriptsize }}

\maketitle

\section{ Introduction}
The main aim of this paper is  to obtain the solution of the integral equation (\ref{eqa1}) involving weighted fractional order integral with respect to another function by using generalized Darbo's $\mathsf{F}\mathcal{P}\mathcal{T}$

\begin{equation}\label{eqa1}
\mathfrak{Y}(\xi)= \mathfrak{P} (\xi,\mathfrak{Y}(\xi))+ \dfrac{w^{-1}(\mathfrak{Y})}{\Gamma(\delta)}\int_{0}^{\xi}\left( \mathcal{U}(\xi)-\mathcal{U}(\eta)\right)^{\delta-1}\mathcal{U}^\prime(\eta)~ w(\eta)~ \mathfrak{S}(\xi,\mathfrak{Y}(\eta))~ d\eta
\end{equation}
    Where $0 \leq \delta < 1, \xi\in I = [0,1]$ ,$\mathfrak{P}$ and $\mathfrak{S} $ are the functions from $I\times \mathcal{R}$ to $\mathcal{R}$ , $\mathcal{U}$ is a function from $I$ to $\mathcal{R}$ ,  Euler's Gamma function is denoted by $\Gamma(,)$ and define as \\ 
\[
\Gamma(\delta) = \int_{0}^{\infty} z^{\delta-1} e^{-z} dz.
\]\\
Also ,~$w$ ~denotes the weighted function and the weighted fractional integral of a function h with respect to the function p of nth order $(n \in \mathbb{N} )$ can be written as
\begin{align}\label{eqz1}
(_{c^{+}}\mathfrak{J}_{w}^{n} h )(z) &= w^{-1}(z) \int_{c}^{z} p^{\prime}(y_{1}) dy_{1}\int_{c}^{y_{1}} p^{\prime} (y_{2}) dy_{2}~ .....~ \int_{c}^{y_{n-1}} w(y_{n}) h(y_{n}) p^{\prime}(y_{n}) dy_{n} \notag \\
&= \dfrac{w^{-1}(z)}{(n-1)!} \int_{c}^{z} (p(z)-p(y))^{n-1} ~w(y)~h(y)~g^{\prime}(y)dy~ , \quad z>c
\end{align}
\par
where $w(z)\neq 0$ is the weighted function , $w^{-1}(z)=\frac{1}{w(z)}$ and p being strictly increasing differentiable function . Then , the corresponding derivative is 
\begin{align}\label{eqz2}
(\mathfrak{D}_{w}^{1}h)(z) &= w^{-1}(z)\dfrac{\mathcal{D}_z}{p^{\prime}(z)}~(w(z)~h(z)).\notag \\
\mathfrak{D}_{w}^{n} h &=\mathfrak{D}_{w}^{1}(\mathfrak{D}_{w}^{n-1}h)\notag \\
&=w^{-1}(z)\left( \dfrac{\mathcal{D}_z }{p^{\prime}(z)}\right)^{n} ~ (w(z)~h(z)) ,
\end{align}
\par
where $\mathcal{D}_z=\frac{d}{dz}$ . The fractional form of the eqations (\ref{eqz1}) and (\ref{eqz2}) are
\begin{equation}\label{eqz3}
(_{c^{+}}\mathfrak{J}_{w}^{\delta} h )(z) = \dfrac{w^{-1}(z)}{\Gamma(\delta)} \int_{c}^{z} (p(z)-p(y))^{\delta-1} ~w(y)~h(y)~g^{\prime}(y)dy~ ,\quad z>c , \quad \delta>0 
\end{equation}
and 
\begin{align}\label{eqz4}
(_{c^{+}}\mathfrak{D}_{w}^{\delta} h)(z)&=(\mathfrak{D}_{w}^{n}  \mathfrak{J}_{w}^{n-\delta} h)(z) \notag \\
&= \dfrac{1}{\Gamma(n-\delta)} \mathfrak{D}_{w}^{n} ~\left( \int_{c}^{z}(p(z)-p(y))^{n-\delta-1}\times w(y)~h(y)~p^{\prime}(y)~dy \right)~,~z>c~,~\delta>0,
\end{align}
respectively, where $ n=\left[ \delta \right] +1$, $\left[ \delta \right]$ being the integer part of $\delta$~(see \cite{fj1}).
\par Moreover , Kuratowski \cite{od2} was the first who introduced  the idea of $M\mathcal{N}\mathsf{C}$ in  1930 which has great importance in $\mathsf{F}\mathcal{P}\mathcal{T}$ . In 1955, G. Darbo \cite{od5} utilized the Kuratowski $M\mathcal{N}\mathsf{C}$ in his construction to generalize the Schauder's fixed point theorem . Following that , many authors analyzed and solved various type of problems in differential equations , integral equations , fractional order integral equations and integrodifferential  equations by using $M\mathcal{N}\mathsf{C}$ . There are several applications of various integral equations which are determined by many reseachers by using fixed point theory and $M\mathcal{N}\mathsf{C}$(see \cite{Arab,ad,ad11,vp,l1,b2,b1,b4,b3}). Recently , many new research paper has been established on the solvability of several kinds of integral and differential equations( see \cite{n1,n2,ad,n4,chaos,n3}, ).\\
\section{Preliminaries}
\par Assume ,  $\left( \mathbb{D} ,\parallel . \parallel \right)$ be a real Banach space and  
$\mathcal{B}(\theta,c_{0})=\left\lbrace t \in \mathbb{D}:\parallel t-\theta \parallel \leq c_{0} \right\rbrace.$ If $\mathfrak{V}(\neq \phi) \subset \mathbb{D}$ . Then, \\
 Let ,
\begin{itemize}
	\item $\mathfrak{M}_{\mathbb{D}}$  denotes the collection of all non-empty bounded subsets of $\mathbb{D}$ and $\mathfrak{N}_{\mathbb{D}}$  denotes the collection of all non-empty relatively compact subsets of $\mathbb{D}$ ,
	\item $\bar{\mathfrak{V}} $ and $Conv\mathfrak{V}$ denote the closure and the convex closure of $\mathfrak{V}$ respectively ,
	\item $\mathcal{R}=(-\infty, \infty),$ \\
	and\item $\mathcal{R}_{+}=\left[ 0, \infty \right) .$
\end{itemize}
  
 Now , we define  $M\mathcal{N}\mathsf{C}$  as given below which is as shown in\cite{Banas1}.
\begin{defn}\label{def}
A map $\psi :\mathfrak{M}_{\mathbb{D}} \rightarrow \mathcal{R}_{+}$ is  known as   $M\mathcal{N}\mathsf{C}$ in $\mathbb{D}.$ If it holds the axioms given below, \begin{enumerate}
\item[(i)]  $\forall ~\mathfrak{V} \in \mathfrak{M}_{\mathbb{D}},$ we get
$\psi(\mathfrak{V})=0$ gives $\mathfrak{V}$ is relatively compact.
\item[(ii)]  ker $\psi = \left\lbrace \mathfrak{V} \in \mathfrak{M}_{\mathbb{D}}: \psi\left( \mathfrak{V}\right)=0   \right\rbrace \neq \phi$ and ker $\psi \subset \mathfrak{N}_{\mathbb{D}}.$
\item[(iii)] $\mathfrak{V} \subseteq \mathfrak{V}_{1} \implies \psi\left( \mathfrak{V}\right) \leq \psi\left( \mathfrak{V}_{1}\right).$
\item[(iv)] $\psi\left( \bar{\mathfrak{V}}\right)=\psi\left( \mathfrak{V}\right).$
\item[(v)] $\psi\left( Conv \mathfrak{V}\right)=\psi\left( \mathfrak{V}\right).$
\item[(vi)] $\psi\left( \mathbb{A} \mathfrak{V} +\left(1- \mathbb{A} \right)\mathfrak{V}_{1} \right) \leq \mathbb{A} \psi\left( \mathfrak{V}\right)+\left(1- \mathbb{A} \right)\psi\left( \mathfrak{V}_{1}\right)$ for $\mathbb{A} \in \left[ 0, 1 \right].$
\item[(vii)] if $\mathfrak{V}_{l} \in \mathfrak{M}_{\mathbb{D}}, \:\mathfrak{V}_{l}= \bar{\mathfrak{V}}_{l}, \: 
\mathfrak{V}_{l+1} \subset \mathfrak{V}_{l}$ for $l=1,2,3,4,...$ and $\lim\limits_{l \rightarrow \infty}\psi\left( \mathfrak{V}_{l}\right)=0$ then $\bigcap_{l=1}^{\infty}\mathfrak{V}_{l} \neq \phi.$ 
\end{enumerate}
\end{defn}

\par The family $ker \psi $ is known as the \textit{kernel of measure} $\psi.$  Since $\psi(\mathfrak{V}_{\infty}) \leq \psi(\mathfrak{V}_{l})$ for any $l,$ we conclude $\psi(\mathfrak{V}_{\infty})=0.$ Then $\mathfrak{V}_{\infty}=\bigcap_{l=1}^{\infty}\mathfrak{V}_{l} \in ker \psi.$

\subsection*{Some useful theorem and definition}

\par
 We recall some fundamental theorems as given below:
 
\begin{thm}(
[Shauder]\cite{sh} )\label{s1}
Let $\mathfrak{X}$   be a nonempty , bounded , closed and convex subset (NBCCS) of a Banach Space $\mathbb{D}.$ Then  $\pounds: \mathfrak{X} \rightarrow \mathfrak{X}$ possesses at least one fixed point provided that $\pounds$ is a compact, continuous mapping.
\end{thm}

\begin{thm}([Darbo]\cite{od5})\label{l1}
For a nonempty , bounded , closed and convex subset (NBCCS) $\mathfrak{X}$ of a Banach Space $\mathbb{D}$ and let $\pounds: \mathfrak{X} \rightarrow \mathfrak{X}.$ Assume that a constant $\varpi\in \left[ 0,1\right) $ such that
\[
\psi(\pounds\mathfrak{H})\leq \varpi~\psi(\mathfrak{H}),\; \mathfrak{H}\subseteq \mathfrak{X}.
\]
 Then $\pounds$ has a fixed point in $\mathfrak{X}$ provided that $\pounds$ is a continuous mapping.
\end{thm}
\par In order to show  an generalization of Darbo's $\mathsf{F}\mathcal{P}\mathcal{T}$, first we have to define the following notions:

\begin{defn}
Let $\alpha: \mathcal{R}_{+} \rightarrow \mathcal{R}_{+}$ be a function which satisfy:
              \[\alpha(x) \geq 1~, \quad\quad x \in \mathcal{R}_{+}.\]
\end{defn}
Also , we denote this class of functions by $\mathfrak{A}$.\\
For example , $\alpha(x)=1+x~$, $x \in \mathcal{R}_{+} $.

\begin{defn}
Let $\mathfrak{B}$ be the collecttion of function $\sigma : \mathcal{R}_{+} \rightarrow [0,1)$
\end{defn}
For example , $ \sigma(z)=\dfrac{z}{1+z}~$, $z\in \mathcal{R}_{+}$.

\section{Fixed point theory}
\begin{thm}\label{th1}
Suppose a nonempty , bounded , closed and convex subset (NBCCS) $\mathbb{L}$  of a Banach space $\mathbb{D}$ and $\mathcal{P}:\mathbb{L} \rightarrow \mathbb{L}$ be a continuous mapping with
\begin{equation}\label{con}
 (\psi(\mathcal{P}\mathcal{G})+l)^{\alpha(\psi(\mathcal{T}\mathcal{G}))} \leq \sigma(\psi(\mathcal{G}))\psi(\mathcal{G})+l ~,\quad l>1 
\end{equation}
 where $\mathcal{G} \subset \mathbb{L}$ and $\psi$ is an arbitrary  $M\mathcal{N}\mathsf{C}$  and $\alpha \in \mathfrak{A}~,\; \sigma\in \mathfrak{B}.$
  Then there exists at least one fixed point for $\mathcal{P}$ in $\mathbb{L}.$
\end{thm}
\begin{proof}
Assume that $\left\lbrace \mathbb{L}_{q}\right\rbrace _{q=1}^{\infty}$ be a sequence with $\mathbb{L}_{1}=\mathbb{L}$
and $\mathbb{L}_{q+1}=Conv(\mathcal{P}\mathbb{L}_{q})$ for $q \in \mathbb{N}.$ Also $\mathcal{P}\mathbb{L}_{1}=\mathcal{P}\mathbb{L} \subseteq \mathbb{L}=\mathbb{L}_{1},\; \mathbb{L}_{2}=Conv(\mathcal{P}\mathbb{L}_{1})\subseteq \mathbb{L}=\mathbb{L}_{1}.$ By proceeding in the same manner gives
$\mathbb{L}_{1} \supseteq \mathbb{L}_{2}\supseteq \mathbb{L}_{3}\supseteq \ldots \supseteq \mathbb{L}_{q}\supseteq \mathbb{L}_{q+1}\supseteq \ldots.$
\par If $\psi(\mathbb{L}_{q_{0}})=0$ for some $q_{0}\in \mathbb{N}$. So $\mathbb{L}_{q_{0}}$ is a compact set.Then we can say that $\mathcal{P}$ has a Fixed Point in $\mathbb{L}.$ (by using Schauder's  theorem )
\par Again, if $\psi(\mathbb{L}_{q})>0,\; q \in \mathbb{N}.$
\par Now, for $q \in \mathbb{N},$
\begin{align*}
&(\psi(\mathbb{L}_{q+1})+l)^{\alpha(\psi(\mathbb{L}_{q+1}))} \\
& =(\psi(Conv(\mathcal{P}\mathbb{L}_{q})+l))^{\alpha(\psi(Conv(\mathcal{P}\mathbb{L}_{q})))} \\
&=(\psi(\mathcal{P}\mathbb{L}_{q})+l)^{\alpha(\psi(\mathcal{P}\mathbb{L}_{q}))} \\
& \leq \sigma(\psi(\mathbb{L}_{q})\psi(\mathbb{L}_{q})+l\\
\end{align*}
Now,\begin{align*}
&\psi(\mathbb{L}_{q+1})+l\leq \sigma(\psi(\mathbb{L}_{q})\psi(\mathbb{L}_{q})+l\\
&\implies\psi(\mathbb{L}_{q+1}) \leq \sigma(\psi(\mathbb{L}_{q})\psi(\mathbb{L}_{q}). \quad\quad \quad[ since,l>0]
\end{align*}

Since, 
\[
\sigma(\psi(\mathbb{L}_{q}))\leq1.
\]
Hence
\[
0\leq\psi(\mathbb{L}_{q+1})\leq\psi(\mathbb{L}_{q}).
\]

Clearly $\left\lbrace \psi(\mathbb{L}_{q})\right\rbrace _{q=1}^{\infty} $ is a decreasing and bounded below sequence .\\
Hence,
\[
\lim\limits_{q \rightarrow \infty} \psi\left( \mathbb{L}_{q}\right)=0.
\]
\par Since $\mathbb{L}_{q}\supseteq \mathbb{L}_{q+1}\quad i.e.\quad\{\mathbb{L}_{q}\} ~ is ~a~ nested~ sequence
,$ in wiew of (vii) of definition \ref{def}, we conclude that $\mathbb{L}_{\infty}=\bigcap_{q=1}^{\infty}\mathbb{L}_{q}$ is non-empty, closed and convex subset of $\mathbb{L}.$ Besides we know that $\mathbb{L}_{\infty}$ belongs to ker$\psi$. So $\mathbb{L}_{\infty}$ is invariant under the mapping $\mathcal{P}$.\\ So, Schauder's fixed point theorem  (theorem \ref{s1} ) implies this result in $\mathbb{L}_{\infty}$,but $\mathbb{L} \subset$ $\mathbb{L}_{\infty}$ , the result is true in $\mathbb{L}$ . Thus ,there exists a fixed point for $\mathcal{P}$  in $\mathbb{L}.$

\end{proof}
\begin{cor}\label{cor2}
Assume that $\mathbb{L}$ be a NBCCS of a Banach space $\mathbb{D}.$ And $\mathcal{P}:\mathbb{L} \rightarrow \mathbb{L}$ be a continuous mapping with
\begin{equation}\label{con1}
\psi(\mathcal{P}\mathcal{G})\leq \varpi~\psi(\mathcal{G}) ~,~ \varpi \in [0,1)
\end{equation}
 where $\mathcal{G} \subset \mathbb{L}$ and $\psi$ be an arbitrary  $M\mathcal{N}\mathsf{C}$ .
  Then $\mathcal{P}$ has fixed point in $\mathbb{L}.$
\end{cor}
\begin{proof} 
 We can get the  Darbo's $\mathsf{F}\mathcal{P}\mathcal{T}$ by taking $\alpha(x)=1$ and $\sigma(x)=\varpi ~, ~x \in \mathcal{R}_{+}$ and $ 0\leq \varpi <1$ in the equation (\ref{con}) of the Theorem \ref{th1} . 
\end{proof}

\section{Measure of noncompactness on $\mathbf{C}([0,1]):$}
  Assume , the space $\mathbb{D}=\mathbf{C}(I)$ is the collection  of all real valued  continuous  operators on $I=[0,1].$ Then , $\mathbb{D}$ is a Banach space with the norm 
\[
\parallel \zeta\parallel=\sup\left\lbrace \left| \zeta(\beta)\right|:\beta \in I \right\rbrace ,\; \zeta\in \mathbb{D}.
\]
 Let $J(\neq \phi) \subseteq \mathbb{D}$ be   bounded. For $\zeta \in J$ with $\vartheta >0,$  the modulus of the continuity of $\zeta$ is denote by $\gamma(\zeta,\vartheta)$ i.e.,
 \[\gamma(\zeta,\vartheta)=\sup \left\lbrace \left| \zeta(\beta_{1})-\zeta(\beta_{2})\right|: \beta_{1},\beta_{2} \in I , \left| \beta_{2}-\beta_{1}\right|\leq \vartheta \right\rbrace.
 \]
   In addition, we define
   \[ \gamma(J,\vartheta)=\sup\left\lbrace \gamma(\zeta,\vartheta): \zeta \in J\right\rbrace;~ \gamma_{0}(J)=\lim\limits_{\vartheta \rightarrow 0}\gamma(J,\vartheta).
  \]
   Where , the function $\gamma_{0}$ is known as $M\mathcal{N}\mathsf{C}$ in $\mathbb{D}$ and the Hausdorff $M\mathcal{N}\mathsf{C}$ ~  $\mathfrak{F}$ is define as $\mathfrak{F}(J)=\frac{1}{2}\gamma_{0}(J)$
  (see \cite{Banas1}) \\
\section{Solvability of fractional integral equation }
\par

In this portion ,  we discuss how our results can be applied to find the solution of an integral equation involving weighted fractional inetegral with respect to another function in Banach space. \par
    Consider the following $\mathsf{F}\mathcal{I}\mathbf{E}$ : 
\begin{equation}\label{eqz5}
\mathfrak{Y}(\xi)= \mathfrak{P} (\xi,\mathfrak{Y}(\xi))+ \dfrac{w^{-1}(\mathfrak{Y})}{\Gamma(\delta)}\int_{0}^{\xi}\left( \mathcal{U}(\xi)-\mathcal{U}(\eta)\right)^{\delta-1}\mathcal{U}^\prime(\eta)~ w(\eta)~ \mathfrak{S}(\xi,\mathfrak{Y}(\eta))~ d\eta
\end{equation}
where $0 \leq \delta < 1, \xi\in I = [0,1]$\\
Let \[
\mathbb{Q}_{e_{0}}=\left\lbrace \mathfrak{Y}\in \mathbb{D}: \parallel \mathfrak{Y} \parallel \leq e_{0} \right\rbrace. \]
Assume that
   \begin{enumerate}
   \item [(I)]
     $\mathfrak{P}:I \times \mathcal{R} \rightarrow \mathcal{R}$ be a continuous and $\exists$  constant $\mathfrak{P}_1 \geq 0$ satisfying
     \[
     \left| \mathfrak{P}(\xi, \mathfrak{Y}(\xi))-\mathfrak{P}(\xi, \mathfrak{Y}_{1}(\xi))\right| \leq \mathfrak{P}_1\left|\mathfrak{Y}(\xi)-\mathfrak{Y}_{1}(\xi) \right| ,
     \]  for $\xi\in I$ and $\mathfrak{Y},\mathfrak{Y}_{1}\in \mathcal{R}.$
 
 Also, \[\hat{\mathfrak{P}}  = \sup_{\xi \in I}|\mathfrak{P}(\xi,0)|. \]
  \item [(II)]
  $\mathfrak{S}:I \times \mathcal{R}\rightarrow \mathcal{R}$ be continuous and $\exists$ a nondecreasing function $ \mathfrak{S}_{1}:\mathcal{R_{+}}\rightarrow \mathcal{R_{+}}$ satisfying , 
  \[
  |\mathfrak{S}(\xi,\mathfrak{Y})|\leq \mathfrak{S}_{1}(\parallel \mathfrak{Y} \parallel)~,~\xi \in I~ ;~ \mathfrak{Y} \in \mathcal{R}. \quad (say)
  \] 
 \item[(III)] $\mathcal{U} : I \rightarrow \mathcal{R}$ and $\mathcal{U}$ is continuously differentiable function.\\
 \item[(IV)] $w$ is bounded so $|w(t)|\leq K_{1}$
~and~
$ |w^{-1}(t)| \leq K_{2}.$\\
\item[(V)] $\exists$ a positive number $ e_{0} $ such that 
\[\mathfrak{P}_{1}e_{0}+\hat{\mathfrak{P}}+\dfrac{K_{1}K_{2}\mathfrak{S}_{1}(e_{0})} {\Gamma(\delta +1)}(\mathcal{U}(1)-\mathcal{U}(0))^{\delta} \leq e_{0}\]
\end{enumerate}

\begin{thm}\label{Tc0} 
 There exists a solution of equation (\ref{eqz5}) in $\mathbb{D}$ whenever conditions  (I)-(V) are satisfied .
\end{thm}

\begin{proof}
  We take the operator $\mathcal{H} : \mathbb{D} \rightarrow \mathbb{D}$ is defined as given below:
  \[(\mathcal{H} \mathfrak{Y})(\xi)= \mathfrak{P} (\xi,\mathfrak{Y}(\xi))+ \dfrac{w^{-1}(\mathfrak{Y})}{\Gamma(\delta)}\int_{0}^{\xi}\left( \mathcal{U}(\xi)-\mathcal{U}(\eta)\right)^{\delta-1}\mathcal{U}^\prime(\eta)~ w(\eta)~ \mathfrak{S}(\xi,\mathfrak{Y}(\eta))~ d\eta
  \]
{\bf Step (1):} Here ,  we will show that the operator $\mathcal{H}$
maps $\mathbb{Q}_{e_0}$ into $\mathbb{Q}_{e_0}$.  Let  $\mathfrak{Y} \in \mathbb{Q}_{e_{0}}$.\\[5pt]
 Now ,we have ,
 \begin{align*}
  & \left| (\mathcal{H} \mathfrak{Y})(\xi))\right|\\
  & \leq  \left|\mathfrak{P} (\xi,\mathfrak{Y}(\xi))\right|+ \left|\dfrac{w^{-1}(\mathfrak{Y})}{\Gamma(\delta)}\int_{0}^{\xi}\left( \mathcal{U}(\xi)-\mathcal{U}(\eta)\right)^{\delta-1}\mathcal{U}^\prime(\eta)~ w(\eta)~ \mathfrak{S}(\xi,\mathfrak{Y}(\eta))~ d\eta \right| \\
  & \leq |\mathfrak{P} (\xi,\mathfrak{Y}(\xi))-\mathfrak{P} (\xi,0)|+|\mathfrak{P} (\xi,0)|+ \left|\dfrac{w^{-1}(\mathfrak{Y})}{\Gamma(\delta)}\int_{0}^{\xi}\left( \mathcal{U}(\xi)-\mathcal{U}(\eta)\right)^{\delta-1}\mathcal{U}^\prime(\eta)~ w(\eta)~ \mathfrak{S}(\xi,\mathfrak{Y}(\eta))~ d\eta  \right|\\
  & \leq \mathfrak{P}_{1} |\mathfrak{Y}(\xi)|+\hat{\mathfrak{P}}+ \left|\dfrac{w^{-1}(\mathfrak{Y})}{\Gamma(\delta)}\int_{0}^{\xi}\left( \mathcal{U}(\xi)-\mathcal{U}(\eta)\right)^{\delta-1}\mathcal{U}^\prime(\eta)~ w(\eta)~ \mathfrak{S}(\xi,\mathfrak{Y}(\eta))~ d\eta  \right| \\
  Also~ , \\
  &\left|\dfrac{w^{-1}(\mathfrak{Y})}{\Gamma(\delta)}\int_{0}^{\xi}\left( \mathcal{U}(\xi)-\mathcal{U}(\eta)\right)^{\delta-1}\mathcal{U}^\prime(\eta)~ w(\eta)~ \mathfrak{S}(\xi,\mathfrak{Y}(\eta))~ d\eta \right|\\
  & \leq \dfrac{|w^{-1}(\mathfrak{Y})|}{\Gamma(\delta)} \left|\int_{0}^{\xi}\left( \mathcal{U}(\xi)-\mathcal{U}(\eta)\right)^{\delta-1}\mathcal{U}^\prime(\eta)~ w(\eta)~ \mathfrak{S}(\xi,\mathfrak{Y}(\eta))~ d\eta \right| \\
  &\leq \dfrac{K_{2}}{\Gamma(\delta)} \int_{0}^{\xi} \left|\left( \mathcal{U}(\xi)-\mathcal{U}(\eta)\right) \right|^{\delta-1}~|\mathcal{U}^\prime(\eta)|~ |w(\eta)|~ |\mathfrak{S}(\xi,\mathfrak{Y}(\eta))|~ d\eta \\
  &\leq \dfrac{K_{1} K_{2} \mathfrak{S}_{1}(\parallel \mathfrak{Y}\parallel)}{\Gamma(\delta)}\int_{0}^{\xi}\left( \mathcal{U}(\xi)-\mathcal{U}(\eta)\right)^{\delta-1}\mathcal{U}^\prime(\eta)~ d\eta \\ 
  &\leq \dfrac{K_{1} K_{2} \mathfrak{S}_{1}(\parallel \mathfrak{Y}\parallel)}{\delta.\Gamma(\delta)}~\left( \mathcal{U}(\xi)-\mathcal{U}(0)\right)^{\delta} \\ 
     &\leq \dfrac{K_{1} K_{2} \mathfrak{S}_{1}(\parallel \mathfrak{Y}\parallel)}{\Gamma(\delta+1)}~\left( \mathcal{U}(1)-\mathcal{U}(0)\right)^{\delta} .\quad \quad \quad[where ,~\delta.\Gamma(\delta)=\Gamma(\delta+1)]
 \end{align*}
  \par Hence $\parallel \mathfrak{Y}\parallel \leq e_{0}$ gives
  \[
  \parallel \mathcal{H}\mathfrak{Y} \parallel \leq \mathfrak{P}_{1}e_{0}+\hat{\mathfrak{P}}+\dfrac{K_{1}K_{2}\mathfrak{S}_{1}(e_{0})} {\Gamma(\delta +1)}(\mathcal{U}(1)-\mathcal{U}(0))^{\delta} \leq e_{0} .
  \]
  Due to the assumption (V) , $\mathcal{H}$ maps $\mathbb{Q}_{e_{0}}$ to $\mathbb{Q}_{e_{0}}$.
 
  {\bf Step (2):} Now , We will show that $\mathcal{H}$ is continuous on $\mathbb{Q}_{e_{0}}.$  Let $\vartheta >0$ and $\mathfrak{Y} , \mathfrak{Y}_{1} \in \mathbb{Q}_{e_{0}}$ such that $\parallel \mathfrak{Y} - \mathfrak{Y}_{1} \parallel < \vartheta $ , for all $\xi \in I.$ \\[5pt]
 Now ,we have , 
  \begin{align*}
  & \left| (\mathcal{H} \mathfrak{Y})(\xi)-(\mathcal{H} \mathfrak{Y}_{1})(\xi)\right| \\
  &  \leq  \left|\mathfrak{P} (\xi,\mathfrak{Y}(\xi))-\mathfrak{P} (\xi,\mathfrak{Y}_{1}(\xi))\right|+ \Big|\dfrac{w^{-1}(\mathfrak{Y})}{\Gamma(\delta)}\int_{0}^{\xi}\left( \mathcal{U}(\xi)-\mathcal{U}(\eta)\right)^{\delta-1}\mathcal{U}^\prime(\eta)~ w(\eta)~ \mathfrak{S}(\xi,\mathfrak{Y}(\eta))~ d\eta  \\
  & - \dfrac{w^{-1}(\mathfrak{Y}_{1})}{\Gamma(\delta)}\int_{0}^{\xi}\left( \mathcal{U}(\xi)-\mathcal{U}(\eta)\right)^{\delta-1}\mathcal{U}^\prime(\eta)~ w(\eta)~ \mathfrak{S}(\xi,\mathfrak{Y}_{1}(\eta))~ d\eta \Big| \\
  &\leq \mathfrak{P}_{1} |\mathfrak{Y}(\xi)-\mathfrak{Y}_{1}(\xi)|+\dfrac{K_{2}}{\Gamma(\delta)}\int_{0}^{\xi}|\left( \mathcal{U}(\xi)-\mathcal{U}(\eta)\right)|^{\delta-1}|\mathcal{U}|^\prime(\eta)~| w(\eta)|~ |\mathfrak{S}(\xi,\mathfrak{Y}(\eta))-\mathfrak{S}(\xi,\mathfrak{Y}_{1}(\eta))|~ d\eta \\
   &\leq \mathfrak{P}_{1} \parallel\mathfrak{Y}-\mathfrak{Y}_{1}\parallel+\dfrac{K_{1} K_{2} \gamma_{e_{0}}(\vartheta)}{\Gamma(\delta)}\int_{0}^{\xi}\left( \mathcal{U}(\xi)-\mathcal{U}(\eta)\right)^{\delta-1}\mathcal{U}^\prime(\eta)~ d\eta\\
   & \leq \mathfrak{P}_{1} \parallel\mathfrak{Y}-\mathfrak{Y}_{1}\parallel+\dfrac{K_{1} K_{2} \gamma_{e_{0}}(\vartheta)}{\Gamma(\delta+1)}\left( \mathcal{U}(1)-\mathcal{U}(0)\right)^{\delta}.
  \end{align*}
  Where ,
 \[ \gamma_{e_{0}}(\vartheta) = \sup { \left\{ |\mathfrak{S}(\xi,\mathfrak{Y}(\eta))-\mathfrak{S}(\xi,\mathfrak{Y}_{1}(\eta))| ~;~ |\mathfrak{Y}-\mathfrak{Y}_{1}| \leq  \vartheta~ ;~ \xi \in I~ ;~ \mathfrak{Y} , \mathfrak{Y}_{1} \in [-e_{0} , e_{0}]\right\} } .\]
 Hence , \quad $\parallel \mathfrak{Y}-\mathfrak{Y}_{1}\parallel < \vartheta  $ \quad gives , 
 \[
 \left| (\mathcal{H} \mathfrak{Y})(\xi)-(\mathcal{H} \mathfrak{Y}_{1})(\xi)\right| < \mathfrak{P}_{1} \vartheta+\dfrac{K_{1} K_{2} \gamma_{e_{0}}(\varrho)}{\Gamma(\delta+1)}\left( \mathcal{U}(1)-\mathcal{U}(0)\right)^{\delta}
 \]
 
i.e. As $\vartheta \rightarrow 0$ we get $\left| (\mathcal{H} \mathfrak{Y})(\xi)-(\mathcal{H} \mathfrak{Y}_{1})(\xi)\right| \rightarrow 0.$ \\ Then, $\mathcal{H}$ is continuous on $\mathbb{Q}_{e_{0}}.$

{\bf Step (3):} An estimate of $\mathcal{H}$ with respect to $\gamma_{0}$. Taking $\varOmega (\neq \phi) \subseteq \mathbb{Q}_{e_{0}}.$ Let  $\vartheta >0$ be arbitrary and choosing $\mathfrak{Y} \in \varOmega$ and $\xi_{1}, \xi_{2} \in I$ such as $\left| \xi_{2}-\xi_{1}\right|\leq \vartheta $ with $\xi_{2} \geq \xi_{1}.$
  \par We have,
    \begin{align*}
  &\left| (\mathcal{H} \mathfrak{Y})(\xi_{2})-(\mathcal{H} \mathfrak{Y})(\xi_{1})\right|\\
  & =  \Big|\mathfrak{P} (\xi_{2},\mathfrak{Y}(\xi_{2}))+ \dfrac{w^{-1}(\mathfrak{Y})}{\Gamma(\delta)}\int_{0}^{\xi_{2}}\left( \mathcal{U}(\xi_{2})-\mathcal{U}(\eta)\right)^{\delta-1}\mathcal{U}^\prime(\eta)~ w(\eta)~ \mathfrak{S}(\xi_{2},\mathfrak{Y}(\eta))~ d\eta  \\   
  &-\mathfrak{P} (\xi_{1},\mathfrak{Y}(\xi_{1}))-\dfrac{w^{-1}(\mathfrak{Y})}{\Gamma(\delta)}\int_{0}^{\xi_{1}}\left( \mathcal{U}(\xi_{1})-\mathcal{U}(\eta)\right)^{\delta-1}\mathcal{U}^\prime(\eta)~ w(\eta)~ \mathfrak{S}(\xi_{1},\mathfrak{Y}(\eta))~ d\eta \Big|\\
  & \leq | \mathfrak{P} (\xi_{2},\mathfrak{Y}(\xi_{2}))-\mathfrak{P} (\xi_{1},\mathfrak{Y}(\xi_{1}))|+ \dfrac{|w^{-1}(\mathfrak{Y})|}{\Gamma(\delta)}\Big(\int_{0}^{\xi_{2}}|\left( \mathcal{U}(\xi_{2})-\mathcal{U}(\eta)\right)|^{\delta-1}\mathcal{U}^\prime(\eta)~ |w(\eta)|~ |\mathfrak{S}(\xi_{2},\mathfrak{Y}(\eta))|~ d\eta\\
  &-\int_{0}^{\xi_{1}}|\left( \mathcal{U}(\xi_{1})-\mathcal{U}(\eta)\right)|^{\delta-1}\mathcal{U}^\prime(\eta)~ |w(\eta)|~ |\mathfrak{S}(\xi_{1},\mathfrak{Y}(\eta))|~ d\eta \Big)\\[10pt]  
  &\leq | \mathfrak{P} (\xi_{2},\mathfrak{Y}(\xi_{2}))-\mathfrak{P} (\xi_{2},\mathfrak{Y}(\xi_{1}))|+| \mathfrak{P} (\xi_{2},\mathfrak{Y}(\xi_{1}))-\mathfrak{P} (\xi_{1},\mathfrak{Y}(\xi_{1}))| \\[10pt]
  &+\dfrac{K_{1}K_{2}\mathfrak{S}_{1}(\parallel \mathfrak{Y}\parallel)}{\Gamma(\delta)} \Big( \int_{0}^{\xi_{2}}\left( \mathcal{U}(\xi_{2})-\mathcal{U}(\eta)\right)^{\delta-1}\mathcal{U}^\prime(\eta)~ d\eta -\int_{0}^{\xi_{1}}\left( \mathcal{U}(\xi_{1})-\mathcal{U}(\eta)\right)^{\delta-1}\mathcal{U}^\prime(\eta)~ d\eta \Big) \\[10pt]
  & \leq~ \mathfrak{P}_{1}~ \gamma(\mathfrak{Y},\vartheta ) + \gamma_{\mathfrak{P}} ( e_{0},\vartheta)+\dfrac{K_{1}K_{2}\mathfrak{S}_{1}(e_{0})}{ \Gamma(\delta+1)} \Big(\left( \mathcal{U}(\xi_{2})-\mathcal{U}(0)\right)^{\delta} - \left( \mathcal{U}(\xi_{1})-\mathcal{U}(0)\right)^{\delta}\Big)  
  \end{align*}

 where 
 \[
 \gamma_{\mathfrak{P}}(e_{0}, \vartheta)=\sup\left\lbrace | \mathfrak{P} (\xi_{2},\mathfrak{Y})-\mathfrak{P} (\xi_{1},\mathfrak{Y})|:\left| \xi_{2}-\xi_{1}\right|\leq \vartheta,\; \xi_{1}, \xi_{2}\in I,\; \parallel \mathfrak{Y} \parallel \leq e_{0} \right\rbrace.
 \]
 And , 
 \[
 \gamma(\mathfrak{Y},\vartheta)=  \sup \left\{ \left| \mathfrak{Y}(\xi_{2})-\mathfrak{Y}(\xi_{1})\right| \leq \vartheta ; |\xi_{2}-\xi_{1}| \leq \vartheta ; \xi_{1},\xi_{2} \in I \right\}
 \]
 As $\vartheta \rightarrow 0 $, then $\xi_{2} \rightarrow \xi_{1}$ , we get 
 \[\lim\limits_{\vartheta \rightarrow 0}\dfrac{K_{1}K_{2}\mathfrak{S}_{1}(e_{0})}{ \Gamma(\delta+1)} \Big(\left( \mathcal{U}(\xi_{2})-\mathcal{U}(0)\right)^{\delta} - \left( \mathcal{U}(\xi_{1})-\mathcal{U}(0)\right)^{\delta}\Big) \rightarrow 0
 \]
 hence, 
 \[\left| (\mathcal{H} \mathfrak{Y})(\xi_{2})-(\mathcal{H} \mathfrak{Y})(\xi_{1})\right| \leq \mathfrak{P}_{1}~ \gamma(\mathfrak{Y},\vartheta ) + \gamma_{\mathfrak{P}} ( e_{0},\vartheta)
 \] 
 i.e 
 \[
 \gamma (\mathcal{H} \mathfrak{Y} , \vartheta )\leq \mathfrak{P}_{1}~ \gamma(\mathfrak{Y},\vartheta ) + \gamma_{\mathfrak{P}} ( e_{0},\vartheta)
 \]
 By uniform continuty of $\mathfrak{P}$ on $I \times [-e_{0} , e_{0}]$ , we have $ \lim\limits_{\vartheta \rightarrow 0}\gamma_{\mathfrak{P}} ( e_{0},\vartheta) \rightarrow 0$ as $ \vartheta \rightarrow 0$ . 
 Taking $\sup_{\mathfrak{Y} \in \varOmega}$ and $ \vartheta \rightarrow 0$ , we get
 \[
 \gamma_{0}( \mathcal{H} \varOmega ) \leq \mathfrak{P}_{1} \gamma_{0} ( \varOmega )
 \]
From Corollary \ref{cor2} , there exists  a  fixed point for $\mathcal{H}$ in $\varOmega \subseteq \mathbb{Q}_{e_{0}}$ \\i,e the equation (\ref{eqz5}) have a solutions in $\mathbb{D}.$\\
Now , there are two suitable examples are given to illustrate the theorem \ref{Tc0}.
 \end{proof}
 \begin{exm}{\rm
 Taking the following $\mathbf{F} \mathcal{I}\mathsf{E}$:
 \begin{equation} \label{eqc1}
   \mathfrak{Y}(\xi) = \dfrac{\mathfrak{Y}+1}{4+\xi^{2}}+ \dfrac{1} {\Gamma(\frac{1} {2})}~ \int_{0}^{\xi}~ \dfrac{ \dfrac{\mathfrak{Y}^{2}(\eta)} {1+\xi^{2}}} {(\xi-\eta)^{\frac{1}{2}}}~ d\eta
 \end{equation}
 for $\xi \in [0,1]=I, $ }  
 \end{exm}
 which is the particular case of equation
 (\ref{eqz5}).\\
 \noindent  Here,\\
\[\mathfrak{P}(\xi,\mathfrak{Y}(\xi)) = \dfrac{\mathfrak{Y}+1}{4+\xi^{2}}~,\] \
\[ \delta = \dfrac{1}{2}~ ,\]\
 \[\mathcal{U}(\mathfrak{Y})= \mathfrak{Y} \quad then \quad \mathcal{U}(\xi)=\xi ~,~ \mathcal{U}(\eta)=\eta \quad and \quad \mathcal{U}^{\prime}(\eta)=1~, \]\
 \[
  w(\mathfrak{Y})=1 \quad then \quad w^{-1}(\mathfrak{Y}) =1 ~,
 \]
 and
 \[
   \mathfrak{S}(\xi,\mathfrak{Y}(\eta))= \dfrac{\mathfrak{Y}^{2}(\eta)}{1+\xi^{2}}.
  \]
Also, It is trivial that $\mathfrak{P}$ is contineous satisfying 
\[
 \left| \mathfrak{P}(\xi,\mathfrak{Y}(\xi)) - \mathfrak{P}(\xi,\mathfrak{Y}_{1}(\xi) \right| \leq \dfrac{|\mathfrak{Y}-\mathfrak{Y}_{1}|}{4}
\] 
Therefore ,~$ \mathfrak{P}_{1} = \dfrac{1}{4} $ \\[5pt]
Now , If  $ \parallel \mathfrak{Y} \parallel \leq e_{0} $  then \\
 \[
 \hat{\mathfrak{P}} = \dfrac{e_{0}}{4},
 \]
 And \\ 
 \begin{align*}
  \left| \mathfrak{S}(\xi,\mathfrak{Y}) \right|& = \left| \dfrac{\mathfrak{Y}^{2}}{1+\xi^{2}} \right|\\[5pt]
 &\leq |\mathfrak{Y}^{2}| \\[5pt]
 &=|\mathfrak{Y}|^{2} \\[5pt]
 & \leq \parallel \mathfrak{Y} \parallel^{2} \\[5pt]
 &= \mathfrak{S}_{1}( \parallel \mathfrak{Y} \parallel)
 \end{align*}
 where , \quad \quad $\mathfrak{S}_{1}(\mathfrak{Y}) = \mathfrak{Y}^{2} $ then $ \mathfrak{S}_{1}(e_{0}) = e_{0}^{2} .$\\
After putting these values , the  inequality of assumption (V) becomes ,
 \begin{align*}
 & \frac{e_{0}}{4}+\frac{e_{0}}{4}+\frac{e_{0}^{2} (1)^{\frac{1}{2}}}{\Gamma(\frac{3}{2})}\leq e_{0}\\
 & \implies \frac{e_{0}}{\Gamma(\frac{3}{2})} \leq \frac{1}{2}\\
 & \implies e_{0} \leq \frac{\Gamma(\frac{3}{2})}{2}.
 \end{align*}
 
 However, assumption (V) is also satisfied for $e_{0} = \dfrac{\Gamma(\frac{3}{2})}{2}.$\\[5pt]
Thus , we have achieved all of the assumptions from $(I)$ to $(V)$ in Theorem \ref{Tc0}~ .\\From Theorem \ref{Tc0} , we can say that  The equation (\ref{eqc1}) have  solutions in  $\mathbb{D}=\mathbb{C}(I).$

\begin{exm}{\rm
 Taking the following $\mathbf{F} \mathcal{I}\mathsf{E}$:
 \begin{equation} \label{eqc2}
   \mathfrak{Y}(\xi) = \dfrac{\mathfrak{Y}+1}{9+\xi^{4}}+ \dfrac{1} {\Gamma(\frac{1} {3})}~ \int_{0}^{\xi}~ \dfrac{\sqrt{ \dfrac{\mathfrak{Y}^{4}(\eta)} {1+\mathfrak{Y}^{4}(\eta)}}} {(\xi-\eta)^{\frac{2}{3}}}~ d\eta
 \end{equation}
 for $\xi \in [0,1]=I, $ }  
 \end{exm}
 which is the particular case of equation
 (\ref{eqz5}).\\
 \noindent  Here,\\
\[\mathfrak{P}(\xi,\mathfrak{Y}(\xi)) = \dfrac{\mathfrak{Y}+1}{9+\xi^{4}}~,\] \
\[ \delta = \dfrac{1}{3}~ ,\]\
 \[\mathcal{U}(\mathfrak{Y})= \mathfrak{Y} \quad then \quad \mathcal{U}(\xi)=\xi ~,~ \mathcal{U}(\eta)=\eta \quad and \quad \mathcal{U}^{\prime}(\eta)=1~, \]\
 \[
  w(\mathfrak{Y})=1 \quad then \quad w^{-1}(\mathfrak{Y}) =1 ~,
 \]
 and
 \[
   \mathfrak{S}(\xi,\mathfrak{Y}(\eta))= \sqrt{\dfrac{\mathfrak{Y}^{4}(\eta)}{1+\mathfrak{Y}^{4}(\eta)}}.
  \]
Also, It is trivial that $\mathfrak{P}$ is contineous satisfying 
\[
 \left| \mathfrak{P}(\xi,\mathfrak{Y}(\xi)) - \mathfrak{P}(\xi,\mathfrak{Y}_{1}(\xi) \right| \leq \dfrac{|\mathfrak{Y}-\mathfrak{Y}_{1}|}{9}
\] 
Therefore ,~$ \mathfrak{P}_{1} = \dfrac{1}{9} $ \\[5pt]
Now , If  $ \parallel \mathfrak{Y} \parallel \leq e_{0} $  then \\
 \[
 \hat{\mathfrak{P}} = \dfrac{e_{0}}{9},
 \]
 And \\ 
 \begin{align*}
  \left| \mathfrak{S}(\xi,\mathfrak{Y}) \right|& = \left| \sqrt{\dfrac{\mathfrak{Y}^{4}}{1+\mathfrak{Y}^{4}}} \right|\\[5pt]
 &\leq |\mathfrak{Y}^{2}| \\[5pt]
 &=|\mathfrak{Y}|^{2} \\[5pt]
 & \leq \parallel \mathfrak{Y} \parallel^{2} \\[5pt]
 &= \mathfrak{S}_{1}( \parallel \mathfrak{Y} \parallel)
 \end{align*}
 where , \quad \quad $\mathfrak{S}_{1}(\mathfrak{Y}) = \mathfrak{Y}^{2} $ then $ \mathfrak{S}_{1}(e_{0}) = e_{0}^{2} .$\\
After putting these values , the  inequality of assumption (V) becomes ,
 \begin{align*}
 & \frac{e_{0}}{9}+\frac{e_{0}}{9}+\frac{e_{0}^{2} (1)^{\frac{1}{3}}}{\Gamma(\frac{4}{3})}\leq e_{0}\\
 & \implies \frac{e_{0}}{\Gamma(\frac{4}{3})} \leq \frac{7}{9}\\
 & \implies e_{0} \leq \frac{7 ~\Gamma(\frac{4}{3})}{9}.
 \end{align*}
 
 However, assumption (V) is also satisfied for $e_{0} = \dfrac{7~ \Gamma(\frac{4}{3})}{9}.$\\[5pt]
Thus , we have achieved all of the assumptions from $(I)$ to $(V)$ in Theorem \ref{Tc0}~ .\\From Theorem \ref{Tc0} , we can say that  The equation (\ref{eqc2}) have  solutions in  $\mathbb{D}=\mathbb{C}(I).$

\section{Declarations}
\subsection*{Funding}
Not applicable.
\subsection*{Conflicts of interests/Competing interests}
The authors declare that they have no competing interests.
\subsection*{Authors’ contributions}
All authors contributed equally to this article. All authors read and approved the final manuscript.

\end{document}